\documentclass[12pt]{amsart}
\usepackage[utf8]{inputenc}

\usepackage{latexsym}
\usepackage{amssymb}
\usepackage{amsmath}
\usepackage{mathtools}

\usepackage{enumerate}
\usepackage{mathrsfs}

\usepackage{tinos}
\usepackage{graphicx}
\usepackage[left=1.9cm, top=2.5cm,bottom=2.5cm,right=1.9cm]{geometry}

\usepackage[dvipsnames]{xcolor}

\definecolor{col1}{rgb}{0.6, 0.7, 0.8}
\definecolor{col2}{rgb}{0.67, 0.75, 0.8}
\definecolor{col3}{rgb}{0.74, 0.8, 0.8}
\definecolor{col4}{rgb}{0.81,0.85, 0.8}
\definecolor{col5}{rgb}{0.98,0.99,0.6}

\definecolor{textcol}{rgb}{0.37,0,0.57}
\usepackage{tikz}
\usepackage{forest}
\forestset{
  default preamble={
   for tree={circle, draw, 
            minimum size=2.1em, 
            inner sep=1pt} 
  }
}

\newtheorem{thm}{Theorem}

\newtheorem{lem}{Lemma}

\begin{document}

\begin{abstract} Let $k$ be a field of characteristic zero. Let $F = X + H$ be a polynomial map from $k^n$ to $k^n$, where $X$ is the identity map and $H$ has only degree two terms and higher. We say that the Jacobian matrix $JH$ of $H$ is strongly nilpotent with index $p$ if for all $X^{(1)},\ldots,X^{(p)} \in k^n$ we have
\begin{align*}
JH(X^{(1)})\ldots JH (X^{(p)}) = 0.
\end{align*}
Every $F$ of this form is a polynomial automorphism, i.e.\ there is a second polynomial map $F^{-1}$ such that $F \circ F^{-1} = F^{-1} \circ F = X$. We prove that the degree of the inverse $F^{-1}$ satisfies
\begin{align*}
\deg(F^{-1}) \leq \deg(F)^{p-1}, 
\end{align*}
improving in the strongly nilpotent case on the well known degree bound $\deg(F^{-1}) \leq \deg(F)^{n-1}$ for general polynomial automorphisms.
\end{abstract}

\title[]{A degree bound for strongly nilpotent polynomial automorphisms}




\author{Samuel G.~G.~Johnston}

\maketitle

\section{Introduction}

\subsection{The Jacobian conjecture}
Let $k$ be a field of characteristic zero, and consider the ring $k[X_1,\ldots,X_n]$ of polynomials in $n$ variables with coefficients in $k$. A polynomial map $F = (F_1,\ldots,F_n)$ is a map $F:k^n \to k^n$ such that each component $F_i$ of $F$ is an element of $k[X_1,\ldots,X_n]$. We define the Jacobian matrix $JF$ of $F$ to be the $n \times n$ matrix whose $(i,j)^{\text{th}}$ entry is given by the element $JF_{i,j} := \frac{ \partial}{\partial X_j} F_i$ of $k[X_1,\ldots,X_n]$. The Jacobian determinant $jF := \det (JF)$ of $F$ is itself an element of the polynomial ring $k[X_1,\ldots,X_n]$. We say that $F$ is a Keller map if its Jacobian determinant is equal to a non-zero constant in $k$. The celebrated Jacobian conjecture asserts that every Keller map is a polynomial automorphism, that is, there exists another polynomial map $F^{-1}$ such that $F \circ F^{-1} = F^{-1} \circ F = X$, where $X$ is the identity map on $k^n$. The Jacobian conjecture is easily verified to be true when the dimension $n$ is equal to $1$ (every Keller map in one dimension is affine), but remains open to this day for all $n \geq 2$. 

We define the degree $\mathrm{deg}(F) := \max_{ 1 \leq i \leq n} \mathrm{deg}(F_i)$ of a polynomial map to be the maximum of the degrees of its components $F_1,\ldots,F_n$ in the variables $X_1,\ldots,X_n$. Numerous authors have studied questions related to whether Keller maps of certains degrees are polynomial automorphisms. Notably, Wang \cite{wang} has shown that every 
quadratic Keller map (i.e.\ of degree at most two) is a polynomial automorphism. While the proof of Wang's theorem is fairly straightforward (see e.g.\ \cite[Page 298]{BCW}), a proof of the cubic case remains elusive. More remarkably still, the Bass, Connell and Wright \cite{BCW} reduction states that if the Jacobian conjecture is true for cubic maps in every dimension, then it is true in general.

A related subject of investigation is how the degrees of polynomial automorphisms and their inverses are related to one another. The central result here is the well known degree bound, which states that whenever $F$ is a polynomial automorphism, it must be the case that the degree of its inverse $F^{-1}$ satisfies
\begin{align} \label{eq:degbound}
\mathrm{deg}(F^{-1}) \leq \mathrm{deg}(F)^{n-1}.
\end{align}
See e.g.\ Theorem 1.5 of \cite{BCW}. 

Suppose $F$ is a polynomial map, and we are seeking to identify whether $F$ is a polynomial automorphism. By composing $F$ with a translation, we may assume without loss of generality that $F(0)=0$. The Jacobian matrix $JF(0)$ of $F$ evaluated at $0 \in k^n$ is an $n \times n$ matrix with entries in $k$ encoding the linear part of $F$. If $F$ is a Keller map, then $JF(0)$ is clearly invertible, and we may precompose $F$ with the (linear) mapping $JH(0)^{-1}:k^n \to k^n$ to obtain a new polynomial map $\tilde{F} := F \circ JH(0)^{-1}$. A brief calculation tells us that $\tilde{F}$ takes the form $\tilde{F} = X + H$, where $X$ is the identity map and $H$ is a polynomial map such that each component $H_i$ of $H$ contains only degree two terms or higher. Clearly $\mathrm{deg}(H) = \mathrm{deg}(\tilde{F}) = \mathrm{deg}(F)$, $\tilde{F}$ is also a Keller map, and $\tilde{F}$ is a polynomial automorphism if and only if $F$ is a polynomial automorphism. Moreover, since $j\tilde{F}(X) $ is equal to a constant, by evaluating $j\tilde{F} = \mathrm{det}(I + JH(X))$ at $X = 0 \in k^n$, we see that this constant must be equal to $1 \in k$. 
In short, in attempting to tackle the Jacobian conjecture, it is sufficient to establish the polynomial invertibility of maps of the form $F = X +H$, where $H$ contains only degree two terms or higher, and $\mathrm{det}( I + JH(X)) =1$ for all $X \in k^n$. In fact, we will see in the next section that it is possible to reduce the problem further.

\subsection{Nilpotent and strongly nilpotent polynomial maps}
We say that a polynomial map $H$ is homogenous of degree $d$ if each component $H_i$ of $H$ is a homogenous polynomial of degree $d$ in the variables $X_1,\ldots,X_n$. Equivalently, each component of $H$ satisfies $H_i(\lambda X) = \lambda^d H_i(X)$ for all $\lambda \in k, X \in k^n$. 

According to the Bass, Connell and Wright reduction \cite{BCW} of the Jacobian conjecture, it is sufficient to establish the polynomial invertibility of every map of the form $F = X + H$, where the Jacobian determinant of $F$ is equal to one, and $H$ is cubic homogenous. Whenever $H$ is homogenous of any degree $d \geq 2$, and we have $\mathrm{det}( I + JH) = 1$, it is straightforward to show that $JH$ is nilpotent, i.e.  
\begin{align} \label{eq:nil}
JH(X)^e = 0 \qquad \text{for every $X$ in $k^n$},
\end{align}
for some $e \leq n$. See for instance Section 2.5 of \cite{JP}. In the sequel we will write $e$ for the least integer $1 \leq e' \leq n$ such that $JH(X)^{e'} = 0$ for every $X$ in $k^n$, and refer to $e$ as the index of nilpotency.

Polynomial maps $F = X + H$ where $H$ has nilpotent Jacobian matrix have been the subject of intense study, with research in this area constituting the seventh chapter of van den Essen's monograph on polynomial automorphisms \cite{EPA}. 
In particular, several authors have wondered whether it is in fact the index $e \leq n$ of nilpotency, rather than the underlying dimension $n$, that controls the degree of the inverse of a polynomial automorphism. Indeed, it was conjectured by Dru\.zkowski and Rusek \cite{DR} that whenever $F = X + H$ where $JH$ is nilpotent with index $e \leq n$, we have the strengthening
\begin{align} \label{eq:DR}
\deg(F^{-1}) \leq^? \deg(F)^{e-1}
\end{align}
of \eqref{eq:degbound}. 

It turns out that \eqref{eq:DR} is false in general; Van den Essen showed in \cite{E} that 
\begin{align*}
H = (3X_4^2X_2 - 2X_3X_4X_5, X_4^2X_5, X_4^3, X_5^3,0,\ldots,0)
\end{align*}
provides a cubic counterexample in dimensions $n \geq 5$. In that same paper however, van den Essen showed that \eqref{eq:DR} holds whenever the $\mathrm{deg}(F) = 3$ and $n \leq 4$. 

It has been observed that for polynomial maps $H$ of low degree and low dimension, the nilpotency condition \eqref{eq:nil} can often be strengthened significantly. In this direction, we say that a Jacobian matrix $JH$ is strongly nilpotent with index $p$ if 
\begin{align} \label{eq:nil2}
JH(X^{(1)}) \ldots JH(X^{(p)}) = 0 \qquad \text{for all $X^{(1)},\ldots,X^{(p)}$ in $k^n$},
\end{align}
and $p$ is the smallest integer with this property. For emphasis we will call the original version of nilpotency in \eqref{eq:nil} \emph{weak nilpotency}, and refer to $e$ as the \emph{weak index} (of nilpotency). If $JH$ is strongly nilpotent with index $p$, then $JH$ is weakly nilpotent with some index $e \leq p$.

It was shown by van den Essen and Hubbers \cite{EH} that every $F = X + H$ with $JH$ strongly nilpotent is a polynomial automorphism. In fact they show that the strong nilpotency of $JH$ is equivalent to $H$ being linearly triangularizable, that is, the existence of a linear map $T:k^n \to k^n$ such that $J(T^{-1}HT)$ is strictly upper triangular. See e.g.\ Theorem 7.4.4 of \cite{EPA} for a proof. As suggested above, it transpires that for certain dimensions and degrees, nilpotency is equivalent to strong nilpotency. Indeed, Meisters and Olech \cite{MO} have shown that whenever $\deg(H) =2$ and $n \leq 4$, $JH$ is weakly nilpotent if and only if $JH$ is strongly nilpotent. Let us remark finally that if $JH$ is strongly nilpotent with some index $p$, then it must be the case that $p \leq n$; see Exercise 
7.4.2 of \cite{EPA}.

\subsection{Main results}
We now state our first main result, which is concerned with the degree of the inverse of strongly nilpotent polynomial automorphisms.

\begin{thm} \label{thm:main}
Let $F = X + H$ be a polynomial map from $k^n$ to $k^n$ such that the Jacobian matrix $JH$ of $H$ is strongly nilpotent with index $p$ (i.e. \eqref{eq:nil2} holds). Then $F$ is a polynomial automorphism, and the degree of the inverse $F^{-1}$ of $F$ satisfies
\begin{align*}
\mathrm{deg}(F^{-1}) \leq \mathrm{deg}(F)^{p-1}.
\end{align*}
\end{thm}

Our approach to proving Theorem \ref{thm:main} is based on studying how the strong nilpotency of $JH$ interacts with the formal expansion for the inverse of $F$. To outline this approach here, we expand our discussion from polynomial maps to power series maps. A power series map is a formal map $F = (F_1,\ldots,F_n)$ whose components are elements of the formal power series ring $k[[X_1,\ldots,X_n]]$. In other words, the components of $F$ take the form
\begin{align*}
F_i = \sum_{ \alpha  } \frac{ F_{i,\alpha}}{ \alpha!} X^\alpha
\end{align*}
where $F_{i,\alpha} \in k$, the sum is taken over all multi-indices $\alpha=(\alpha_1,\ldots,\alpha_n) \in \mathbb{Z}_{ \geq 0}^n$, $\alpha! := \alpha_1! \ldots \alpha_n!$, and $X^\alpha := X_1^{\alpha_1} \ldots X_n^{\alpha_n}$. Two formal power series maps may be composed to form a new formal power series map. A power series map is a polynomial map if only finitely many of the coefficients $\{F_{i,\alpha} : 1 \leq i \leq n, \alpha \in \mathbb{Z}_{ \geq 0}^n \}$ are nonzero. Consider now $F = X - H$ of the form 
\begin{align} \label{eq:power}
F_i = X_i - \sum_{ |\alpha| \geq 2}\frac{ H_{i,\alpha}}{ \alpha!} X^\alpha,
\end{align}
where the sum is taken over all $\alpha \in \mathbb{Z}_{ \geq 0}^n$ such that $|\alpha| = |\alpha_1| + \ldots +|\alpha_n| \geq 2$. (The choice of minus sign in \eqref{eq:power} simplifies various formulas in the sequel.) Every $F$ of the form in \eqref{eq:power} has a compositional inverse power series map $F^{-1}$ of the form $F^{-1} = X_i +  \sum_{ |\alpha| \geq 2}\frac{ G_{i,\alpha}}{ \alpha!} X^\alpha$. The Jacobian conjecture is precisely the statement that the inverse power series map of every Keller map of the form \eqref{eq:power} is not just a power series map, but in fact a polynomial map. 

Given $F = X - H$ of the form \eqref{eq:power}, it is possible to give a formula for the coefficients $G_{i,\alpha}$ of the power series map inverse of $F$ in terms of a sum over weighted rooted trees, in which the weights of these trees depend on the coefficients $H_{i,\alpha}$ \cite{JP}. Sketching this formula here (we will provide more details in the Section \ref{sec:proof}), we let $\mathbb{S}_{i,\alpha}$ denote the set of finite rooted trees in which:
\begin{itemize}
\item Every vertex is either a leaf or has two or more children.
\item Every vertex of the tree has a \emph{type} in $\{1,\ldots,n\}$. The root has type $i$ and the sum of the leaf types is equal to $\alpha$ (so that in particular the number of leaves is equal to $|\alpha|=\alpha_1+\ldots+\alpha_n$).
\item The leaves of the same types have additional labels distinguishing themselves from one another.
\end{itemize}
The inversion formula \cite[Theorem 2.3]{JP} then states that 
\begin{align} \label{eq:Gform}
G_{i,\alpha} = \sum_{ \mathcal{T} \in \mathbb{S}_{i,\alpha}}\mathcal{E}_H(\mathcal{T}) \qquad \text{where} \qquad \mathcal{E}_H(\mathcal{T} ) :=  \prod_{ v \in \mathcal{I}} H_{\tau(v),\mu(v)} \in k,
\end{align}
and the product $\mathcal{E}_H(\mathcal{T})$ is taken over the set $\mathcal{I}$ of internal vertices of $\mathcal{T}$, $\tau(v)$ is the type of $v$, and $\mu(v)$ is the multi-index counting the types of the children of $v$. We remark that if $H$ has degree $d$, then $H_{i,\beta} = 0$ whenever $|\beta|  = \beta_1 + \ldots + \beta_n > d$. In particular, if $H$ has degree $d$ then the sum in \eqref{eq:Gform} is supported on the trees in $\mathbb{S}_{i,\alpha}$ in which every vertex has at most $d$ children.

Suppose that we have a polynomial map $F = X - H$ and we wish to verify that the inverse power series map $F^{-1} = X+ G$ of $F$ is in fact a polynomial map. The inversion formula \eqref{eq:Gform} gives us a concrete approach to this problem: all we need to do is show that for all $|\alpha|$ sufficiently large, the sum in \eqref{eq:Gform} is zero. It transpires that the strong nilpotency of $JH$ can be utilised to ensure that this is the case.

Before stating our second main result, Theorem \ref{thm:power}, which is concerned with specialising the inversion formula \eqref{eq:Gform} for strongly nilpotent $JH$, we remark that for any power series map $F$ we may define the Jacobian matrix of $JF$ of $F$ to be the matrix whose $(i,j)^{\text{th}}$ entry is given by the element
\begin{align} \label{eq:entry}
JF_{i,j} =  \frac{ \partial}{\partial X_j} F_i = \sum_{ \alpha} \frac{ F_{i,\alpha+\mathbf{e}_j} }{ \alpha!} X^\alpha,
\end{align}
of $k[[X_1,\ldots,X_n]]$, where $\mathbf{e}_j = (0,\ldots,0,1,0,\ldots,0)$ is the multi-index with a $1$ in the $j^{\text{th}}$ slot and zeroes elsewhere. As in the case where $F$ is a polynomial map, we may speak plainly of the nilpotency and strong nilpotency of the Jacobian matrix of a power series map $F$ with components in the power series ring $k[[X_1,\ldots,X_n]]$. 
 
Given a labelled rooted tree, we refer to the set of vertices a graph distance $k$ away from the root as generation $k$, or its $k^{\text{th}}$ generation. Clearly then the $0^{\text{th}}$ generation comprises only the root. Given a labelled rooted tree $\mathcal{T}$ in $\mathbb{S}_{i,\alpha}$, the height of $\mathcal{T}$, written $\mathrm{height}(\mathcal{T})$, is the largest integer $k$ such that generation $k$ of $\mathcal{T}$ is non-empty. 

We will see in Lemma \ref{lem:fern} that if the Jacobian matrix of $H$ is strongly nilpotent with index $p$, then certain sums of energies $\mathcal{E}_H(\mathcal{T})$ (given in \eqref{eq:Gform}) over a class of tree called \emph{ferns of height $p$} are zero. Using the product structure of the energy $\mathcal{E}_H(\mathcal{T})$, in Section \ref{sec:proof} we show that the set of trees in $\mathbb{S}_{i,\alpha}$ 
of height at least $p$ may be partitioned into equivalence classes in such a way that the sum of $\mathcal{E}_H(\mathcal{T})$ over each of these classes is equal to zero.  As a result, we are able to prove the following result, which says that whenever $H$ is strongly nilpotent with index $p$, the sum in \eqref{eq:Gform} is supported on trees whose $p^{\text{th}}$ generation is empty:

\begin{thm} \label{thm:power}
Let $F = X  - H$ be a power series map with components of the form $H_i = \sum_{|\alpha| \geq 2} \frac{H_{i,\alpha}}{\alpha!} X^\alpha$ and write $F^{-1}_i = X_i + \sum_{ |\alpha| \geq 2} \frac{ G_{i,\alpha}}{\alpha!} X^\alpha$ for the components of the inverse of $F$. If $JH$ is strongly nilpotent with index of nilpotency $p$, then we have the refinement
\begin{align} \label{eq:Gform2}
G_{i,\alpha} =   \sum_{\substack{ \mathcal{T} \in \mathbb{S}_{i,\alpha}\\
                                 \mathrm{height}(\mathcal{T}) \leq p-1}} \mathcal{E}_H(\mathcal{T})
\end{align}
of \eqref{eq:Gform}.
\end{thm}

We now show that Theorem \ref{thm:main} is a consequence of Theorem \ref{thm:power}.

\begin{proof}[Proof of Theorem \ref{thm:main} assuming Theorem \ref{thm:power}]
Let $F = X - H$ be a polynomial map such that $H$ has degree $d$ and $JH$ is strongly nilpotent with index $p$. Then according to Theorem \ref{thm:power}, if $F^{-1} = X + G$ is the inverse power series map of $F = X - H$, then the components of $G$ are given by $G_i = \sum_{ |\alpha| \geq 2} \frac{G_{i,\alpha}}{\alpha!}X^\alpha$ with $G_{i,\alpha}$ given in \eqref{eq:Gform2}. 

\color{black}
On the other hand, as remarked under \eqref{eq:Gform}, since $H$ has degree $d$, the sum in \eqref{eq:Gform2} is supported on trees in which every internal vertex has at most $d$ children. Any tree with height at most $p-1$ and in which every internal vertex has at most $d$ children may have at most $d^{p-1}$ leaves. Since $G_{i,\alpha}$ is a sum over trees with $|\alpha| = \alpha_1 + \ldots + \alpha_n$ leaves, it follows from \eqref{eq:Gform2} that $G_{i,\alpha}$ may only be nonzero when $|\alpha| \leq d^{p-1}$. In other words, 
\begin{align*}
\mathrm{deg}(F^{-1}) \leq \mathrm{deg}(F)^{p-1},
\end{align*}
and in particular, the inverse power series map $F^{-1}$ of $F$ is in fact a polynomial map, completing the proof of Theorem \ref{thm:main}. 

\end{proof}

\color{black}

Before proceeding with the proof of Theorem \ref{thm:power} in Section \ref{sec:proof}, let us close the introduction by briefly touching on the works of Singer \cite{singer1,singer2}, who uses combinatorial arguments to study the invertibility of polynomial maps, concentrating mainly on the quadratic case. Let us highlight \cite{singer1} in particular, in which Singer gives an ingenious combinatorial argument involving ranking Otter trees (\cite{otter}) to establish that whenever $F = X - H$ with $H$ quadratic (i.e.\ $d=2$) and $JH^3 =0$ (that is, weakly nilpotent with index $p=3$), we have
\begin{align*}
\mathrm{deg}(F^{-1}) \leq 6,
\end{align*}
regardless of the dimension $n$. This bound is sharp; see Example 7.4.15 of \cite{EPA}. Singer's argument involves showing that the vast majority of binary trees have the property that they are the least member of a \emph{shuffle class} of size at most $3!$. Unfortunately, it is not immediately clear how to generalise his approach to tackle the case $d=2, JH^e = 0$ for $e \geq 4$.

\vspace{5mm}
That concludes the introduction, in the next (and final) section, we supply a more detailed account of the inversion formula \eqref{eq:Gform}, state and prove the strong fern lemma (Lemma \ref{lem:fern}), and thereafter prove Theorem \ref{thm:power}.

\section{Proof of Theorem \ref{thm:power} } \label{sec:proof}

We begin in Section \ref{sec:formal}, which is lifted entirely from \cite{JP}, by stating the inversion formula for formal power series maps. In the following Section \ref{sec:strongfern} we state and prove a variant of a result by van den Essen and Hubbers \cite{EH} characterising strong nilpotency in terms of sums of certain trees called \emph{ferns}. Finally, in Section \ref{sec:mainproof} we tie our arguments together to prove Theorem \ref{thm:power}. 

\subsection{The formal inversion formula} \label{sec:formal}
We begin by expanding further on the inversion formula outlined in \eqref{eq:Gform}. We follow the exact formulation given in \cite{JP}, though we mention here that the formula we use is essentially equivalent to earlier inversion formulas due to Wright \cite{wright} and Haiman and Schmidt \cite{HS}. 

A rooted tree is a finite graph-theoretic tree $(V,E)$ (with vertex set $V$ and edge set $E$) with a designated root vertex $v_0$ in $V$. 
For $m \geq 0$, let $V_m$ denote the set of vertices whose graph distance from the root is $m$. We refer to $V_m$ as generation $m$ of the tree, and note that the $0^{\text{th}}$ generation $V_0 = \{ v_0 \}$ consists of just the root $v_0$. Each vertex $w$ in $V_{m+1}$ has a unique parent vertex $v$ in $V_{m}$, and in this case we say $w$ is a child of $v$. We say a vertex is a leaf if it has no children, and we say it is internal if it has at least one child. The vertex set may be partitioned into $V = \mathcal{I} \sqcup \mathcal{L}$, where $\mathcal{I}$ the set of internal vertices and $\mathcal{L}$ is the set of leaves. The height of $(V,E)$, writting $\mathrm{height}(V,E)$, is the maximum $m$ such that $V_m$ is non-empty. We say a rooted tree is \emph{proper} if every internal vertex has at least two children. For the remainder of the paper, the word \emph{tree} will refer to a proper rooted tree.

We will consider labellings of trees. Define $[n] := \{1,\ldots,n\}$ and for a multi-index $\alpha = (\alpha_1,\ldots,\alpha_n)$ of non-negative integers define the set $[\alpha]$
\begin{align*}
[\alpha] := \{ (j,k) \in \mathbb{Z}^2 : 1 \leq k \leq \alpha_j \}.
\end{align*}
We will consider trees whose leaves are in bijection with $[\alpha]$. More specifically, let $\mathbb{S}_{i,\alpha}$ denote the set of quadruplets $\mathcal{T} := (V,E,\tau,\phi)$ with the following properties:
\begin{itemize}
\item The pair $(V,E)$ denotes a proper rooted tree.
\item The labelling function $\phi:\mathcal{L} \to [\alpha]$ is a bijection between the leaves $\mathcal{L}$ of $(V,E)$ and the set $[\alpha]$, giving each leaf a \emph{label} in $[\alpha]$.
\item The typing function $\tau:V \to [n]$ gives each vertex of the tree a \emph{type} in $[n]$ according to certain rules: it gives the root type $i$, and the type given to each leaf respects the label of that leaf, in the sense that that whenever $\phi(v) = (j,k)$ for some $k \in \{1,2,\ldots\}$, we have $\tau(v) = j$. We therefore think of the leaf $(j,k)$ as the $k^{\text{th}}$ leaf of type $j$. There are no constraints on the types of non-root internal vertices. 
\end{itemize}

\begin{figure}[ht!]
  \centering   
 \begin{forest}
[1, fill=col1
[2,  fill=col2
[{4,1}, fill=col4]
[{3,2},fill=col3]
]
[3,fill=col3
[{3,1},  fill=col3]
[{1,1},  fill=col1]
[{1,2}, fill=col1]
]
[{1,3},  fill=col1]
]
\end{forest}
  \caption{An element $\mathbb{S}_{i,\alpha}$ where $i = 1$ and $\alpha = (3,0,2,1)$. A vertex labelled $(j,k)$ has type $j$. The colours of the vertices correspond to their types.}\label{fig:treefinal}
\end{figure}
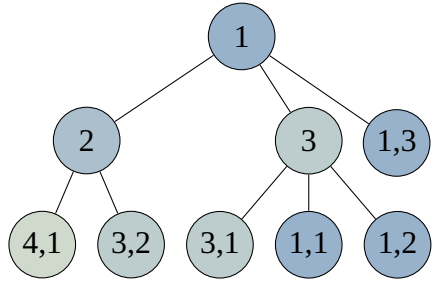

We say two labelled trees  $(V,E,\tau,\phi)$ and $(V',E',\tau',\phi')$ are \emph{isomorphic} if there is a bijection $\Psi:V \to V'$ between the underlying vertex sets preserving all of the structure of the tree: $\left( \Psi(u), \Psi(v) \right) \in E'$ if and only if $(u,v) \in E$, $\tau'(\Psi(u)) = \tau(u)$ and $\phi'(\Psi(u) ) = \phi(u)$. Whenever we speak of a set or collection of trees, technically we mean a set or collection of isomorphism classes according to this equivalence. 

So in summary, $\mathbb{S}_{i,\alpha}$ is the set of (isomorphism classes of) finite proper rooted trees whose leaves are labelled by $[\alpha]$, and whose vertices have types in $\{1,\ldots,n\}$ in such a way that the root has type $i$ and the leaves have type equal to the first coordinate of their label.

For an internal vertex $v$ of a labelled tree $\mathcal{T}$ in $\mathbb{S}_{i,\alpha}$, the outdegree $\mu(v)$ of $v$ is defined to be the multi-index whose $j^{\text{th}}$ component is given by 
\begin{align} \label{eq:outdegree}
\mu(v)_j := \# \{ w \in V : \text{$w$ is a child of $v$ and $\tau(w) = j$} \}.
\end{align}
Given a proper labelled tree $\mathcal{T} \in \mathbb{S}_{i,\alpha}$ and a collection $H := \left( H_{j,\alpha}  : 1 \leq j \leq n, \alpha \in \mathbb{Z}_{\geq 0}^n : |\alpha| \geq 2 \right)$ of elements of $k$, we define the $H$-energy of the tree $\mathcal{T}$ to be the element of $k$ given by 
\begin{align} \label{eq:energy}
\mathcal{E}_H \left( \mathcal{T}  \right) := \prod_{ v \in \mathcal{I} }H_{ \tau(v) , \mu(v) } .
\end{align}
\color{black}
The following result from \cite{JP} is a more precise statement of the inversion formula \eqref{eq:Gform}, stating that the coefficients of the inverse of a power series map may be given in terms of sums of tree energies over proper trees.

\begin{thm}[Theorem 2.3 of \cite{JP}] \label{thm:inversion}
Let $F$ be an element of $k[[X_1,\ldots,X_n]]^n$ such that the $i^{\text{th}}$ component of $F$ takes the form
\begin{align*}
F_i(X) = X_i - \sum_{ |\alpha| \geq 2 } \frac{ H_{i,\alpha}}{\alpha!} X^\alpha
\end{align*}
for some coefficients $(H_{i,\alpha})$. Then the $i^{\text{th}}$-component of the inverse $G$ of $F$ is given by $G_i(X) = X_i +   \sum_{ |\alpha| \geq 2 } \frac{ G_{i,\alpha}}{\alpha!} X^\alpha$, where for $|\alpha| \geq 2$,
\begin{align} \label{eq:GDEF}
G_{i, \alpha}  = \sum_{ \mathcal{T} \in \mathbb{S}_{i,\alpha} } \mathcal{E}_H \left( \mathcal{T}  \right).
\end{align}

\end{thm}

\subsection{The strong fern lemma} \label{sec:strongfern}

Suppose $H$ is a power series map whose components take the form $H_i = \sum_{ |\alpha| \geq 2 } \frac{ H_{i,\alpha}}{\alpha!} X^\alpha$, and such that the the Jacobian matrix of $H$ is strongly nilpotent. Our next result, which is fundamentally identical to a result in \cite{EH}, describes the combinatorial implications of the strong nilpotency of $JH$. 

\begin{lem} \label{lem:fern}
Let $H$ be a power series map with components $H_{i,\alpha} = \sum_{ |\alpha| \geq 2 } \frac{ H_{i,\alpha}}{\alpha!} X^\alpha$. Suppose $JH$ is strongly nilpotent with index $p$. Then for any $i,j \in [n]$, and for any multi-indices $\alpha^{(1)},\ldots,\alpha^{(p)}$ in $\mathbb{Z}_{ \geq 0}^n$, we have 
\begin{align} \label{eq:fernsum}
\sum_{ 1 \leq k_0,\ldots,k_{p} \leq n } \mathrm{1}_{\{k_0 =i, k_p = j\}}  \prod_{ \ell=1}^p H_{k_{\ell-1},\alpha^{(\ell)}+\mathbf{e}_{k_\ell} } = 0,
\end{align} 
where
\begin{align*}
 \mathrm{1}_{\{k_0 =i, k_p = j\}}   =
\begin{cases}
1 \qquad \text{if $k_0=i$ and $k_p =j$}\\
0 \qquad \text{otherwise}.
\end{cases}
\end{align*}
\end{lem}

Let us explain why we call Lemma \ref{lem:fern} the \emph{strong fern lemma}.
A \emph{fern} of height $p$ is a rooted tree of height $p$ with a designated \emph{sink} vertex in its $p^{\text{th}}$ generation, and with the property that for each $0 \leq \ell \leq p-1$, the $\ell^{\text{th}}$ generation contains exactly one internal vertex. 
The sums occuring in \eqref{eq:fernsum} may be then be regarded as an energy sum $\sum \mathcal{E}_H(\mathcal{T})$, where the sum is taken over the set of labelled ferns of height $p$ with root of type $i$, sink of type $j$, and with leaf types $\alpha^{(1)},\ldots,\alpha^{(p)}$. By leaf types $(\alpha^{(1)},\ldots,\alpha^{(p)})$, we mean that for $0 \leq \ell \leq p-1$ the leaves in the $\ell^{\text{th}}$ generation have types $\alpha^{(\ell)}$, whereas the leaves in the final generation have types $\alpha^{(p)} + \mathbf{e}_j$.

\color{black}
\begin{proof}[Proof of lemma \ref{lem:fern}]
The strong nilpotency of $JH$ tells us that for any $X^{(1)},\ldots,X^{(p)}$, the $n \times n$ matrix of power series $JH(X^{(1)})\ldots JH(X^{(p)})$ is equal to zero. Studying the $(i,j)^{\text{th}}$ entry of this matrix product, we see that
\begin{align} \label{eq:orange}
 \left( JH(X^{(1)})\ldots JH(X^{(p)}) \right)_{i,j} = \sum_{ 1 \leq k_0,\ldots,k_{p} \leq n } \mathrm{1}_{\{k_0 =i, k_p = j\}} \prod_{\ell=1}^p JH(X^{(\ell)})_{k_{\ell-1},k_\ell},
\end{align} 
which is itself a power series in the variables $\{ X^{(\ell)}_k : 1 \leq \ell \leq p, 1 \leq k \leq n \}$, is equal to zero. 

Using \eqref{eq:entry} in \eqref{eq:orange} we have 
\begin{align*}
0 = \left( JH(X^{(1)})\ldots JH(X^{(p)}) \right)_{i,j} = \sum_{ 1 \leq k_0,\ldots,k_{p} \leq n } \mathrm{1}_{\{k_0 =i, k_p = j\}} \prod_{\ell=1}^p \sum_{ |\alpha^{(\ell)}| \geq 1} \frac{ H_{k_{\ell-1},\alpha^{(\ell)}+\mathbf{e}_{k_\ell} }}{\alpha^{(\ell)}!} (X^{(\ell)})^{ \alpha^{(l)} } .
\end{align*} 
By examining the coefficient of the monomial $\prod_{ \ell = 1}^p ( X^{(\ell)})^{ \alpha^{(l)} } = \prod_{ \ell = 1}^p \prod_{ m = 1}^n (X^{(\ell)}_m )^{\alpha^{(\ell)}_m}$, we see that for all $\alpha^{(1)},\ldots,\alpha^{(\ell)}$ we have 
\begin{align*}
\sum_{ 1 \leq k_0,\ldots,k_{p} \leq n } \mathrm{1}_{\{k_0 =i, k_p = j\}}  \prod_{ \ell=1}^p H_{k_{\ell-1},\alpha^{(\ell)}+\mathbf{e}_{k_\ell} } = 0,
\end{align*} 
completing the proof of Lemma \ref{lem:fern}.
\end{proof}

Lemma \ref{lem:fern} has an analogue for when the Jacobian matrix $JH$ is weakly nilpotent with index $e$ (as opposed to strongly nilpotent with index $p$). This analogue invovles an energy sum $\mathcal{E}_H(\mathcal{T})$ over a far larger set of labelled ferns. More specifically, in the weakly nilpotent case, we fix a single multi-index $\alpha$ and sum over all ferns of height $e$ whose leaftype across all generations sum to $\alpha$; in the strongly nilpotent case, we fix $p$ multi-indices $\alpha^{(1)},\ldots,\alpha^{(p)}$, and sum over ferns such that the leaf types in generation $\ell$ are given by $\alpha^{(\ell)}$. In both cases we have a root of type $i$ and a sink of type $j$. We refer the reader to Lemma 2.8 of \cite{JP} for the weak case in our notation, thought we mention that the idea is present in earlier works by Wright \cite{wright} and Singer \cite{singer1,singer2}.

\subsection{Proof of Theorem \ref{thm:power}} \label{sec:mainproof}

\begin{proof}[Proof of Theorem \ref{thm:power}]
Let $F = X - H$ be as in the statement of Theorem \ref{thm:power}. Then according to Theorem \ref{thm:inversion}, the inverse $F^{-1} = X + G$ of $F$ has coefficients 
\begin{align*} 
G_{i, \alpha}  = \sum_{ \mathcal{T} \in \mathbb{S}_{i,\alpha} } \mathcal{E}_H \left( \mathcal{T}  \right).
\end{align*}
We now use the strong nilpotency of $JH$ and Lemma \ref{lem:fern} to show that this sum is supported on trees with height at most $p$. Indeed, write
\begin{align} \label{eq:GDEFX}
G_{i, \alpha}  = \sum_{\substack{ \mathcal{T} \in \mathbb{S}_{i,\alpha}\\
                                 \mathrm{height}(\mathcal{T}) \leq p-1}} \mathcal{E}_H(\mathcal{T})  +   \sum_{\substack{ \mathcal{T} \in \mathbb{S}_{i,\alpha}\\
                                 \mathrm{height}(\mathcal{T}) \geq p}} \mathcal{E}_H(\mathcal{T}).
\end{align}
Recall that for any tree $\mathcal{T}$ in $\mathbb{S}_{i,\alpha}$, the leaves of $\mathcal{T}$ are in bijection with $[\alpha] := \{ (j,k) \in \mathbb{Z}^2 : 1 \leq k \leq \alpha_j \}$. Order the elements of $[\alpha]$ (that is, the leaves) lexicographically. (Any other ordering of $[\alpha]$ would do just as well.) Suppose $\mathrm{height}(\mathcal{T}) \geq p$, so that $\mathcal{T}$ has at least one leaf in generation $p$ or higher. Choose the leaf $v^*$ in generation $p$ or higher that has the greatest label in the lexicographic ordering of $[\alpha]$. Since $v^*$ lies in generation $p$ or higher, we may define a sequence of vertices $(v_0,\ldots,v_p)$ in $\mathcal{T}$ by letting $v_p = v^*$, and letting $v_{\ell-1}$ be the parent of $v_\ell$ for every $1\leq \ell \leq p$. We call the sequence of vertices $(v_0,v_1,\ldots,v_p)$ the \emph{spine} of $\mathcal{T}$. Every tree in $\{\mathcal{T} \in \mathbb{S}_{i,\alpha} : \mathrm{height}(\mathcal{T}) \geq p \}$ has a canonical spine associated with the leaf in generation $p$ or higher that is highest in the lexicographic ordering of $[\alpha]$.

We now define an equivalence relation on $\{\mathcal{T} \in \mathbb{S}_{i,\alpha} : \mathrm{height}(\mathcal{T}) \geq p \}$, by declaring $\mathcal{T}$ and $\mathcal{T}'$ to be equivalent if they are identical but for the types $\tau(v_1),\ldots,\tau(v_{p-1})$ of their spines vertices $v_1,\ldots,v_{p-1}$. In other words, $\mathcal{T} \sim \mathcal{T}'$ if $\mathcal{T} = (V,E,\tau,\phi)$ and $\mathcal{T}' = (V',E',\tau', \phi')$, and there is a graph isomorphism $\Psi:V \to V'$ such that $\tau'(\Psi(v)) = \tau(v)$ and $\phi'(\Psi(v)) = \phi(v)$ for all $v \neq v_1,\ldots,v_{p-1}$. It is clear that this relation is indeed an equivalence relation. Each equivalence class contains exactly $n^{p-1}$ members, corresponding to the $n^{p-1}$ different ways of colouring the vertices $v_1,\ldots,v_{p-1}$ with colours in $\{1,\ldots,n\}$.

We will now show using Lemma \ref{lem:fern} that for any equivalence class $[\mathcal{T}]$ in $\{\mathcal{T} \in \mathbb{S}_{i,\alpha} : \mathrm{height}(\mathcal{T}) \geq p \}$ we have
\begin{align} \label{eq:fernclass}
\sum_{ \mathcal{T}' \in [\mathcal{T}]} \mathcal{E}_H(\mathcal{T}') = 0.
\end{align}
To see that \eqref{eq:fernclass} holds, let $\mathcal{T} \in \{ \mathcal{T} \in \mathbb{S}_{i,\alpha} : \mathrm{height}(\mathcal{T}) \geq p \}$, and let $(v_0,\ldots,v_p)$ be its spine. Let $\tau(v_0) = i, \tau(v_p)= j$, and for $1 \leq \ell \leq p$, recalling the notation \eqref{eq:outdegree} define $\alpha^{(\ell)} := \mu(v_{\ell-1}) - e_{\tau(v_{\ell})}$ to be the type sum of the children of $v_{\ell-1}$ lying off the spine. (Here, $\tau(v_\ell)$ is the type of $v_\ell$, and $e_{\tau(v_\ell)}$ is the multi-index with a $1$ in the $\tau(v_\ell)^{\text{th}}$ slot, with zeroes in the other slots.)

With the energy product \eqref{eq:energy} in mind, define the product $\lambda_{[\mathcal{T}]}$ in $k$ over terms off-the-spine in $\mathcal{T}$ by 
\begin{align*}
\lambda_{[\mathcal{T}]} = \prod_{v \notin \{v_0,v_1,\ldots,v_{p-1}\}} H_{\tau(v),\mu(v)},
\end{align*}
where the product is taken over all vertices in $\mathcal{T}$ not contained in the spine. Now since every element $\mathcal{T}'$ agrees with $\mathcal{T}$ everywhere but for the types of $(v_1,\ldots,v_{p-1})$, if the types of $v_1,\ldots,v_{p-1}$ in $\mathcal{T}'$ are given by $\tau(v_1)=k_1,\ldots,\tau(v_{p-1})=k_{p-1}$ we have 
\begin{align*}
\mathcal{E}_H(\mathcal{T}' ) = \lambda_{[\mathcal{T}]} \prod_{ \ell=1}^p H_{k_{\ell-1},\alpha^{(\ell)} + \mathbf{e}_{k_{\ell}} },
\end{align*}
where we are using the conventions $k_0 = i$ and $k_\ell = j$. 

In particular, we have
\begin{align*}
\sum_{ \mathcal{T}' \in [\mathcal{T}]} \mathcal{E}_H(\mathcal{T}')  = \lambda_{[\mathcal{T}]} \sum_{ 1 \leq k_0,\ldots,k_{p} \leq n }\prod_{ \ell=1}^p H_{k_{\ell-1},\alpha^{(\ell)} + \mathbf{e}_{k_{\ell}} } = 0,
\end{align*}
where the latter sum vanishes due to Lemma \ref{lem:fern}. That establishes \eqref{eq:fernclass}.

Since the equivalence classes partition $\{ \mathcal{T} \in \mathbb{S}_{i,\alpha} : \mathrm{height}(\mathcal{T}) \geq p \}$, we may partition the sum\\
 $  \sum_{\substack{ \mathcal{T} \in \mathbb{S}_{i,\alpha}    \mathrm{height}(\mathcal{T}) \geq p}} \mathcal{E}_H(\mathcal{T})$ into terms whose contribution is zero. It follows that the latter term in the expansion \eqref{eq:GDEFX} is zero, thereby completing the proof of Theorem \ref{thm:power}.
\end{proof}

\section*{Acknowledgements}
This research is supported by the EPSRC funded Project EP/S036202/1 \emph{Random fragmentation-coalescence processes out of equilibrium}.


\end{document}